\documentclass [a4paper,11pt]{article}

\usepackage[italian, english]{babel} 
\usepackage[T1]{fontenc} 
\usepackage{lmodern}
\usepackage{textcomp}

\usepackage[utf8]{inputenc}

\newsavebox{\tbox}
\usepackage{lineno}
\usepackage{amsfonts}
\usepackage{amssymb}
\usepackage{amsmath}
\usepackage{verbatim}
\usepackage{amsthm}
\usepackage{graphicx}
\usepackage{hyperref}
\usepackage{xcolor}
\hypersetup{
    colorlinks,
    linkcolor={blue!50!black},
    citecolor={red!50!black},
    urlcolor={red!50!black}
}

\title{Generating alternating and symmetric groups with two elements of fixed order}
\author{Daniele Garzoni}
\date{}

\newcommand{\Addresses}{{% additional braces for segregating \footnotesize
  \bigskip
  \footnotesize

  Daniele Garzoni, \textsc{Università degli Studi di Padova, Dipartimento di Matematica "Tullio Levi-Civita", Via Trieste 63, 35121 Padova, Italy}\par\nopagebreak
  \textit{E-mail address}: \texttt{daniele.garzoni@phd.unipd.it}

}}

\newtheorem {thm}{Theorem}[section]
\numberwithin{thm}{section}

\theoremstyle{remark}

\newtheorem*{notat}{Notation}

\theoremstyle{remark}
\newtheorem{cla}{Claim}

\newenvironment{sistemino}
    {\left\{\begin{matrix}}
    {\end{matrix}\right.}

\theoremstyle{remark}
\newtheorem*{prof}{Proof of Theorem 2.3 in case (A)}
\newtheorem*{profo}{Proof of Theorem 2.3 in case (B)}

\theoremstyle{definition}

\newcommand{\gen}[1]{\ensuremath{\langle #1\rangle}}

\newcommand{\floor}[1]{\ensuremath{\lfloor #1\rfloor}}

\begin{document}

\maketitle

\begin{abstract}

We answer a question raised by Lanier \cite{Lan} about the possibility of generating $A_n$ and $S_n$ with two elements of order $k$, where $n \geqslant k \geqslant 3$. We show that this can always be done apart from some clear exceptions.
\end{abstract}

\section{Introduction}

Results about generating sets of groups with order constraints have a long history. For alternating and symmetric groups, they go back to the beginning of the 20th century with a paper of Miller \cite{M2}. Miller showed that, except for some cases with $n \leqslant 8$, $A_n$ is generated by an element of order $2$ together with an element of order $k=3$. He later \cite M extended the result to the case $k \geqslant 4$, and he discussed the same question for symmetric groups. Results of similar flavour can be found for instance in \cite{Nu}.

Recently Lanier \cite{Lan} addressed the problem of generating mapping class groups of surfaces with a small number of elements of fixed order $k \geqslant 5$. He showed that, if the genus of the surface is sufficiently large, then $3$ elements suffice, except for the case $k=5$, where he showed that $4$ elements suffice. He then considered the same question for alternating and symmetric groups. The author finally combined the results for mapping class groups and alternating groups in order to obtain similar results for some automorphism groups of free groups and for $\text{SL}(n,\mathbb{Z})$. 

Concerning alternating and symmetric groups, Lanier showed that if $n \geqslant k \geqslant 3$, then $3$ elements of order $k$ suffice to generate $S_n$ if $k$ is even and to generate $A_n$ if $k$ is odd. He then conjectured that $2$ elements should in fact be enough. Moreover, he showed that, if $3 \leqslant k \leqslant n-2$ with $k$ even, $4$ elements of order $k$ are sufficient in order to generate $A_n$.

There are some obvious restrictions in this problem. First of all, if $k$ is odd then the elements of order $k$ lie in $A_n$, so that they cannot generate $S_n$. Moreover, if $k \in \{n-1,n\}$ is a power of $2$, then the elements of order $k$ coincide with the $k$-cycles: that, being odd permutations, cannot generate $A_n$. Finally, for $n=6$ the only elements of order $6$ are the $6$-cycles and the elements with cycle-shape $(2,3,1)$, both of which are odd permutations and cannot generate $A_6$.

In this note we show that no other restriction exists. In other words, we confirm the conjecture raised by Lanier and we improve his result for $A_n$ and $k$ even, including in our analysis also the case $k \in \{n-1,n\}$.

In order to make the statement of the result more concise, we introduce some terminology. Let $k$ be a positive integer, and let $G$ be a group. When $G$ is generated by finitely many elements of order $k$, we denote with $d_k(G)$ the smallest cardinality of a generating set of $G$ in which all elements have order $k$.

\begin{thm}
\label{mainthm}
Let $n \geqslant k \geqslant 3$ be natural numbers. The following hold:
\begin{itemize}
\item[(i)] if $k$ is even, then $d_k(S_n)=2$. Moreover, $d_k(A_n)=2$ unless $k \in \{n-1,n\}$ is a power of $2$ or $(n,k)=(6,6)$.
\item[(ii)] if $k$ is odd, then $d_k(A_n)=2$ (unless $n=3$ in which case $d_3(A_3)=1$).
\end{itemize}
\end{thm}

\begin{comment}
Note that, if $k$ is odd, then an element of order $k$ in $S_n$ is an even permutation. Hence, $S_n$ cannot be generated by elements of order $k$ with $k$ odd. On the other hand, if $k \in \{n-1, n-2\}$ is a power of $2$, then the only elements of order $k$ in $S_n$ are the $k$-cycles: which, being odd permutations, cannot generate $A_n$. Moreover, in $S_6$ the only elements of order $6$ are the $6$-cycles and the elements with cycle type $(2,3,1)$: again, they cannot generate $A_6$. Therefore, we need to prove only the positive part of Theorem \ref{mainthm}.
\end{comment}

Consider now the free group $F_n$ on $n$ letters. There exists a map $\text{Aut}(F_n) \rightarrow \text{GL}(n,\mathbb{Z})$ induced by the projection $F_n \rightarrow F_n/(F_n)' \cong \mathbb{Z}^n$. The preimage of $\text{SL}(n,\mathbb{Z})$, that we denote with $\text{Aut}^{+}(F_n)$, is called the special automorphism group of $F_n$. Starting with $\text{Out}(F_n)=\text{Aut}(F_n)/\text{Inn}(F_n)$, with the same construction we can define $\text{Out}^{+}(F_n)$.

Theorem \ref{mainthm} gives immediately the following improvement of \cite{Lan} Theorem 1.3 (see the proof of \cite{Lan} Theorem 1.3).

\begin{thm} Let $n$ and $k$ be natural numbers, with $k \geqslant 5$ and $n \geqslant 2(k-1)$, and let $G$ be among $\text{Aut}^{+}(F_n)$, $\text{Out}^{+}(F_n)$ and $\text{SL}(n,\mathbb{Z})$. Then, $d_k(G) \leqslant 6$. If in addition $k \geqslant 6$ and $n \geqslant 2k$, then $d_k(G) \leqslant 5$.
\end{thm}

% It is our hope that Theorem \ref{mainthm} may have other consequences in the generation of infinite groups.

\begin{comment}
\begin{notat} The underlined permuted set is $\{1, \ldots, n \}$. We adopt the notation used in \cite{Lan}. For every $1 \leqslant i \leqslant n$, we denote with $h_i(a)$ a step $i$-cycle starting at $a$, namely, the cycle $(a, a+1, \ldots, a+i-1)$, where all the entries are viewed mod $n$. We set moreover $s_i (a,l)= \prod_{m=0}^{l-1} {h_i(a + im)}$, where all the entries are viewed mod $n$ and $1 \leqslant l \leqslant \lfloor n/i \rfloor$.

Throughout, the natural number $j$ is defined by the conditions $j \in \{0, \ldots, k-1\}$ and $n \equiv j$\: mod $k$.

\end{notat}

\end{comment}

\section*{Acknowledgements}

I would like to thank László Pyber, for bringing \cite{Lan} to my attention; and Justin Lanier, for some comments on his paper.

\section{Proof of Theorem \ref{mainthm}}

The proof will be divided in two cases: $k \leqslant n/2$ and $k > n/2$. In every case, we will find two elements $a$ and $b$ of order $k$ such that $\gen{a,b}$ is primitive. At this point, in the literature there are several results --- some of them going back to the 19th century --- that ensure that a given primitive group $G$ with some additional properties must be either alternating or symmetric.

For instance, if $G$ is $2$-transitive and contains an element with sufficiently small support, then it is either alternating or symmetric; this is due to Bochert \cite{Bo}. Manning \cite{Ma} proved a similar statement replacing $2$-transitivity with $6$-transitivity (see also \cite{Wi2}, 16 p. 42). In the same spirit, Wielandt \cite{Wi} showed that if $G$ is $t$-transitive and $t$ is large enough depending on $n$, then $G$ is either alternating or symmetric. For us these results are enough; we summarise them in the following theorem. Recall that the support of a permutation is the set of points moved by the permutation.

\begin{thm}
\label{old}
Let $G$ be a primitive permutation group of degree $n$. Then, $G$ contains $A_n$ provided one of the following conditions is satisfied:
\begin{itemize}
\item[(i)] $G$ is $2$-transitive and contains a nontrivial element $g$ with $|\text{supp}(g)| < n/3-2\sqrt n/3$.
\item[(ii)] $G$ is $6$-transitive and contains a nontrivial element $g$ with $|\text{supp}(g)| < 3n/5$.
\item[(iii)] $G$ is $t$-transitive with $t \geqslant 3 \log n$ (here and elsewhere, logs in base $2$).
\end{itemize}
\end{thm}

The previous conditions require $G$ to be multiply transitive. An old result of Marggraff \cite{Mar}, with a short proof by Levingston and Taylor \cite{LT}, gives a method to obtain this property.
\begin{thm}
\label{livin}
Let $G$ be a primitive permutation group of degree $n$. If $G$ contains a cycle fixing $m$ points, then $G$ is $(m+1)$-transitive.
\end{thm}

We remark that Theorem \ref{old} can be dramatically strenthened using the CFSG. For instance, using the Schreier Conjecture O'Nan \cite{Ona1} \cite{Ona2} showed that in point (iii) it is possible to take $t=6$. Point (ii) becomes therefore useless. Moreover, combining Theorem \ref{livin} with the classification of $4$-transitive groups, it turns out that a primitive group containing a cycle fixing at least $3$ points is either alternating or symmetric (see for instance \cite J). For our purposes, however, these improvements would help only in small degree ($n \leqslant 90$), so that it does not seem the case to rely on them.

\begin{comment}
The common strategy is to find two elements $a$ and $b$ of order $k$ such that $\gen{a,b}$ is primitive and contains a cycle fixing at least 3 points. In view of the following theorem, that can be found in \cite{J}, this will be sufficient in order to conclude that $\gen{a,b}$ contains $A_n$.

\begin{thm}
\label{cicli}
A primitive group of degree $n$ containing a cycle fixing at least $3$ points is either $A_n$ or $S_n$.
\end{thm}

The previous theorem relies on the Classification of Finite Simple Groups. However, we we will use this result quite superficially: we could avoid to rely on it if only we checked singularly some cases of small degree (see Remark ???).

There is an old theorem of Marggraff \cite{Ma} which asserts that a primitive permutation group containing a cycle fixing $m$ points is $(m+1)$-transitive. Almost a century later, Livingston and Taylor \cite{LT} gave a short proof of this result. Combining this therem with the classification of $4$-transitive groups (that follows from the CFSG) one obtain Theorem \ref{cicli}. 
\end{comment}

Once we have found $a$ and $b$ such that $\gen{a,b}$ is primitive and satisfies one of the conditions of Theorem \ref{old}, in order to conclude the proof of Theorem \ref{mainthm} it will be sufficient to control the parity of $a$ and $b$. More precisely, we will be able to choose $a$ and $b$ both odd (in which case $\gen{a,b}=S_n$) and both even (in which case $\gen{a,b}=A_n$).

It seems that we are proving too much choosing $a$ and $b$ both odd. However, the same pair $(a,b)$ will be even for some values of $n$ and $k$, and will be odd for some other values: so that, essentially, we will prove the statement for $A_n$ and $S_n$ all at once (see the proof of Claim \ref{clai1} and the table after it). 

As this may help in the exposition, we restate Theorem \ref{mainthm} in view of the last considerations. We say that an element of $S_n$ has \textit{sign $+1$} (resp. \textit{sign $-1$}) if it is an even permutation (resp. odd permutation).

\begin{thm}
\label{mainprop}
Consider the set $\Omega$ of all triples $(n,k,\mathfrak s)$ such that $n \geqslant k \geqslant 3$ are natural numbers, $\mathfrak s \in \{+1, -1\}$, and the following conditions are satisfied:
\begin{itemize}
\item[(i)] if $k$ is odd, then $\mathfrak s=+1$.
\item[(ii)] if $k \in \{n-1, n\}$ is a power of $2$, then $\mathfrak s=-1$.
\item[(iii)] $(6,6,+1) \not\in \Omega$.
\end{itemize}
Then, for every $\omega=(n,k,\mathfrak s) \in \Omega$, $S_n$ contains two elements $a_{\omega}$ and $b_{\omega}$ of order $k$, of sign $\mathfrak s$, such that $\gen{a_{\omega},b_{\omega}}$ is primitive and satisfies one of the conditions of Theorem \ref{old}.
\end{thm}

Rather then dividing neatly all the cases, the strategy is to start with two somehow reasonable $a$ and $b$, and to modify them suitably whenever they do not work. With no risk of confusion, along the proofs the dependence on $\omega$ will be omitted: so that our elements will be $a$ and $b$. Also the symbol $\mathfrak s$ will often be omitted.

\subsection{The case $k \leqslant n/2$}

This is the case that will require the most computations. Here we will always apply Theorem \ref{livin} in order to use condition (iii) in Theorem \ref{old}. In other words, we will show that $\gen{a,b}$ contains a cycle fixing at least $\lceil 3 \log n -1 \rceil$ points. The cycle we are looking for will always be $a^{-1}b$.

We will first obtain $a$ and $b$ in every case, showing that $a^{-1}b$ is a cycle, but not caring about the number of points it fixes. Then, once all the $a$'s and $b$'s will be defined, we will make sure that enough points are fixed; and finally we will show that $\gen {a,b}$ does not preserve any system of blocks. This is perhaps logically backwards, but the primitivity is the longest part, while the number of fixed points is fast, so that we prefere to present the proof in this order. Throughout, when we write the cycle decomposition of an element of $S_n$ we usually avoid to write the fixed points.

\begin{cla}
\label{clai1}
For every $\omega=(n,k,\mathfrak s) \in \Omega$ with $k \leqslant n/2$, there exist $a_{\omega}$ and $b_{\omega}$ of order $k$, of sign $\mathfrak s$, such that $\gen{a_{\omega},b_{\omega}}$ is transitive and $a_{\omega}^{-1}b_{\omega}$ is a cycle.
\end{cla}
\begin{proof}

Consider $a=(1, \ldots, k)(k+1, \ldots, 2k)\cdots ((\floor{n/k}-1)k+1, \ldots, \floor{n/k}k)$ and $b$ the translation of $a$ on the right of $k-1$ positions, that is, $b=(k, \ldots, 2k-1)(2k, \ldots, 3k-1)\cdots (\floor{n/k}k, \ldots, (\floor{n/k}+1)k-1)$, where all the entries are viewed mod $n$. Throughout, these will be called the \textit{original} $a$ and $b$. It is clear that $\gen {a,b}$ is transitive. Moreover, unless specifically pointed out, any modification we will perform on $a$ and $b$ will only (possibly) unify some of their orbits, so that transitivity will not be compromised. Let $j$ be defined by the conditions $j \in \{0, \ldots, k-1\}$ and $n \equiv j$\: mod $k$.

 If $j=k-1$, direct computation shows that $a^{-1}b$ is a cycle: $a^{-1}b=(1,k+1,2k+1, \ldots, \floor{n/k}k+1, \floor{n/k}k+2, \ldots, n, \floor{n/k}k, (\floor{n/k}-1)k, \ldots, k, k-1, \ldots, 2)$, where what happens in the dots should be clear. We choose these $a$ and $b$ when $j=k-1$ and the sign is good. Note that if $k$ is odd then $\mathfrak s = +1$; hence certainly we choose these $a$ and $b$ when $k$ is odd. In order to change sign, we multiply $a$ and $b$ respectively by $(\floor{n/k}k+1, \floor{n/k}k+2$) and $(k-2,k-1)$. Note that the support of these transpositions is disjoint from the support of, respectively, $a$ and $b$, so that the new elements still have order $k$. Moreover, the result of this modification on $a^{-1}b$ amounts to multiplying it on the left by $(\floor{n/k}k+1, \floor{n/k}k+2$) and on the right by  $(k-2,k-1)$. Since the points of each transposition are consecutive points in the cycle $a^{-1}b$, the resulting element is the same cycle, in which though $k-1$ and $\floor{n/k}k+2$ have disappeared.

If $j \neq k-1$, the (original) $a^{-1}b$ is not a cycle; it is instead the product of 2 cycles of equal length (plus 1-cycles). However, it is easy to turn this element into a cycle: it is sufficient to multiply it on the right by $(k, k+1, k+2)$. This is equivalent to multiplying $b$ on the right by $(k, k+1, k+2)$, that is, to swapping $k$ and $k+1$ in the first cycle of $b$. It is easy to check that the resulting cycle is $(k+1, 2k+1, \ldots, \floor{n/k}k+1, \floor{n/k}k+2, \ldots, n, 1, k+2, k,  k-1, \ldots, k-j, \floor{n/k}k, (\floor{n/k}-1)k, \ldots, 2k)$.\footnote{If $j=1$, $\floor{n/k}k+1=n$ and it is meant that $\floor{n/k}k+2$ does not appear in the writing. In the same way, if $j=0$ it is meant that $k-1$ does not appear; and if $\floor{n/k}=2$ it is meant that $(\floor{n/k}-1)k$ does not appear.} We choose these $a$ and $b$ when $j \neq k-1$ and the sign is good (as above, certainly when $k$ is odd). In order to change sign, we divide the cases $j \in \{0,1\}$ and $j \not\in \{0,1\}$.

If $j \not\in \{0,1\}$, we may act exactly as in the case $j=k-1$. If $j \in \{0,1\}$, however, there is no space for multiplying by a transposition. We then go back to the original $a$ and $b$, and split the last cycle of the original $a$ and the last cycle of the original $b$ in two parts of length $k/2$. Note that, since $j \in \{0,1\}$, the first $(k/2)$-cycle of $b$ contains $n$. It is easy then to see that $\gen{a,b}$ is transitive.

In fact, the first $(k/2)$-cycle of $b$ contains 1 unless $(k,j)=(4,1)$. This feature implies different behaviors. In particular, if $(k,j)=(4,1)$ then $a^{-1}b$ is a cycle: $a^{-1}b=(1,5,\ldots, n-4, n-2, n, n-1, n-5, \ldots, 4, 3)$, where the steps in the dots again should be clear. In the other cases, instead, $a^{-1}b$ is a product of 2 cycles of equal length (plus 1-cycles). As we have already done, multiplying on the right by $(k, k+1, k+2)$ turns it into a cycle: $a^{-1}b=(k+1, 2k+1, \ldots, (\floor{n/k}-1)k+1, (\floor{n/k}-1/2)k+1, \floor{n/k}k+1, 1, k+2, k, k-j, (\floor{n/k}+1/2)k, \floor{n/k}k, (\floor{n/k}-1)k, \ldots, 2k)$ (where if $\floor{n/k}=2$ then $2k+1$ and $(\floor{n/k}-1)k$ do not appear; and if $j=0$ then $\floor{n/k}k+1$ and $k-j$ do not appear).
\end{proof}

Although perhaps not strictly necessary, for the convenience of the reader we list $a$ and $b$ in the different cases.

Here (and only here) we adopt the notation used in \cite{Lan}, as it may help in compacting writings. For every $1 \leqslant i \leqslant n$, we denote with $h_i(a)$ a step $i$-cycle starting at $a$, namely, the cycle $(a, a+1, \ldots, a+i-1)$, where all the entries are viewed mod $n$. We set moreover $s_i (a,l)= \prod_{m=0}^{l-1} {h_i(a + im)}$, where all the entries are viewed mod $n$ and $1 \leqslant l \leqslant \lfloor n/i \rfloor$. For instance, the original $a$ and $b$, that were defined in the proof of Claim \ref{clai1}, are respectively $s_k (1,\floor{n/k})$ and $s_k (k,\floor{n/k})$. Since we will write these elements several times below, we take the chance to further compact notations and we call them respectively $a_0$ and $b_0$. Consider the following pair of $a$'s and $b$'s (the meaning of the conditions on $j$ are explained below; recall that $j \in \{0, \ldots, k-1\}$ and $n \equiv j$\: mod $k$).

\begin{itemize}

\item[(i)] $a= a_0$

$b=b_0$ \qquad\hspace*{\fill}$j=k-1$

\item[(ii)] $a=a_0$

$b=b_0\cdot s_3(k,1)$ \qquad\hspace*{\fill}$j \neq k-1$

\item[(iii)] $a=a_0 \cdot s_2(\floor{n/k}k+1,1)$

$b=b_0 \cdot s_2(k-2,1)$ \qquad\hspace*{\fill}$j=k-1$

\item[(iv)] $a=a_0 \cdot s_2(\floor{n/k}k+1,1)$

$b=b_0 \cdot s_2(k-2,1) \cdot s_3(k,1)$ \qquad\hspace*{\fill}$j \not\in \{0,1,k-1\}$

\item[(v)] $a=s_k (1,\floor{n/k}-1) \cdot s_{k/2} ((\floor{n/k}-1)k+1, 2)$

$b=s_k (k,\floor{n/k}-1)\cdot  s_{k/2} (\floor{n/k}k, 2) \cdot s_3(k,1)$ \qquad\hspace*{\fill}$j \in \{0,1\},$ 

\qquad\hspace*{\fill} $(k,j) \neq (4,1)$

\item[(vi)] $a=s_k (1,\floor{n/k}-1) \cdot s_{k/2} ((\floor{n/k}-1)k+1, 2)$

$b=s_k (k,\floor{n/k}-1)\cdot  s_{k/2} (\floor{n/k}k, 2)$ \qquad\hspace*{\fill}$(k,j) = (4,1)$

\end{itemize}

In each pair, $a$ and $b$ have the same sign. Moreover, elements in (i) - (ii) have the same sign, opposite with respect to elements in (iii) - (vi). In order to obtain a generating pair we proceed as follows. Given $G_n \in \{A_n, S_n\}$, we look at the pairs which consist of even permutations if $G_n=A_n$, and of odd permutations if $S_n=A_n$. As already noticed in the proof of Claim \ref{clai1}, we look at (iii) - (vi) only if $k$ is even. Indeed, if $k$ is odd then $G_n=A_n$ and the even elements are (i) - (ii).

Once we have chosen among the two parts (i) - (ii) and (iii) - (vi), there is a unique pair depending on the congruence of $n$ mod $k$: that is the generating pair we are looking for. Actually, the cases $(n,k, \mathfrak s)=(6,3,+1),(7,3,+1),(8,3,+1)$ are exceptions that will be handled separately.

It is of course possible to give a concrete description of the different cases. More precisely, if $G_n=A_n$ we must look at (i) - (ii) if and only if either $k$ is odd or $k$ and $\floor{n/k}$ are both even. Conversely, if $G_n=S_n$ we must look at (i) - (ii) if and only if $k$ is even and $\floor{n/k}$ is odd.

We go on with the proof. Now we look at the number of points fixed by the cycle $a^{-1}b$.

\begin{cla}
\label{claim2}
For every $\omega=(n,k,\mathfrak s) \in \Omega$, with $k \leqslant n/2$ and $n \geqslant 90$, $a_{\omega}^{-1}b_{\omega}$ fixes at least $\lceil 3 \log n -1 \rceil$ points. If $n < 90$, the thesis of Theorem \ref{mainprop} is satisfied. It is necessary to change $a_{\omega}$ and $b_{\omega}$ only in the cases $(6,3,+1), (7,3,+1)$ and $(8,3,+1)$.

\end{cla}
\begin{proof} Consider the supports of the $k$-cycles of the original $a$, that is, $c_i = \{ik+1, ik+2, \ldots, (i+1)k\}$, $i=0, \ldots, \floor{n/k}-1$. It is easy to check that, except for $c_0, c_1$ and $c_{\floor{n/k}-1}$, each of the other $c_i$'s contains $k-2$ points that are fixed by $a^{-1}b$ in every case (namely, all the entries but those $\equiv 0, 1$\: mod $k$). Moreover, in $c_1$ there are at most $4$ points not fixed by $a^{-1}b$ (the maximum is attained when $\floor{n/k}=2$ and the second cycle of $a$ has been splitted, in which case $k+1, k+2, k+k/2+1, 2k$ are not fixed by $a^{-1}b$). Therefore, $a^{-1}b$ fixes at least 
\[m:= (\floor{n/k}-3)(k-2)+(k-4)\]
points, where the first term appears only if $\floor{n/k} \geqslant 3$ and the last term only if $k \geqslant 4$. Since either $k \geqslant \sqrt n$ or $\floor{n/k} \geqslant \sqrt n -1$, it is clear that $m$ is larger than $3 \log n -1$ provided $n$ is large enough. Now we get some more concrete estimates.

Assume $k=3$. Then $\floor{n/3} \geqslant n/3-1$ and it is easy to check that, for $n \geqslant 90$, $m$ is larger than $3 \log n-1$. In the same way we can check that the same holds for $k=4,5$. On the other hand, if $\floor{n/k}=2$ then $k \geqslant n/3$ and $m \geqslant 3 \log n -1$. In the same way we can check that the same holds for $\floor{n/k}=3,4$ (for $\floor{n/k}=3$, $n \geqslant 90$ is sharp). Hence we may assume $k \geqslant 6$ and $\floor{n/k} \geqslant 5$. Now either $k \geqslant \sqrt n$ or $\floor{n/k} \geqslant \sqrt n -1$. In both cases, $n \geqslant 90$ is enough in order to have $m \geqslant 3 \log n -1$.

For $n \leqslant 89$, it is possible to check that, except for the cases $(6,3,+1), \\ (7,3,+1)$ and $(8,3,+1)$, the $a$'s and $b$'s that we defined work: they generate $A_n$ when $\mathfrak s=+1$ and $S_n$ when $\mathfrak s=-1$. This can be done using a computer (or cheating and using the CFSG --- there are very few cases in which $a^{-1}b$ fixes less than $3$ points! --- see after Theorem \ref{livin}).

In the three remaining cases, we modify $b$. In the case $(6,3,+1)$ we may take $b=(3,4,5)$; in the case $(7,3,+1)$ we may take $b=(1,2,7)(3,4,5)$; and in the case $(8,3,+1)$ we may take $b=(1,6,7)(2,5,8)$.

\end{proof}

\begin{cla} For every $\omega=(n,k,\mathfrak s) \in \Omega$, with $k \leqslant n/2$, $\gen{a_{\omega},b_{\omega}}$ is primitive.
\label{clpri}
\end{cla}

\begin{proof} Since in every case $\gen{a,b}$ is transitive, there remains to show that $\gen{a,b}$ does not preserve any system of blocks. Consider $t$ a nontrivial divisor of $n$, and consider a system of blocks in which each block has cardinality $t$. Assume this system is preserved by some $\gen{a,b}$; we want to get to a contradiction. Throughout the proof, for every $i \in \{1, \ldots, n \}$ we indicate with $B_i$ the block containing $i$. We also write the action of $S_n$ on $ \{1, \ldots, n \}$ on the right. Recall that if a permutation $g$ preserves a system of blocks, if $c$ is a cycle in the cycle decomposition of $g$, and if $\Lambda_1, \ldots, \Lambda_r$ are the blocks that have nonempty intersection with $\text{supp}(c)$, then $|\Lambda_i \cap \text{supp}(c)| = |\Lambda_j \cap \text{supp}(c)|$ for every $i,j$, and $|\text{supp}(c)|=r \cdot |\Lambda_1 \cap \text{supp}(c)|$. In the following, when we mention cases (i) - (vi) we refer to those after Claim \ref{clai1}.

Note that the transitivity of $\gen{a,b}$ implies that there does not exist a block fixed by both $a$ and $b$. It is easy to see that, for every $(n,k,\mathfrak s)$, every $k$-cycle of $a$ either moves a fixed point of $b$ or moves some point in the same way as $b$ does.\footnote{With a unique exception, namely, when $k=3$ and we are in case (ii). Then $(4,5,6)$ sends $4$ to $5$, as $b^{-1}$ --- and not $b$ --- does. For our argument there is no difference.} The same sentence holds true if we swap $a$ and $b$. Therefore, the support of a $k$-cycle of $a$ or $b$ intersects at least 2 blocks. The same holds for the transpositions $(\floor{n/k}k+1, \floor{n/k}k+2)$ and $(k-2,k-1)$ that we used to change parity to $a$ and $b$. Now $a$ and $b$ might have also $(k/2)$-cycles (cases (v), (vi)). If $k \geqslant 6$, the same consideration as above holds. If $k=4$, however, direct check shows that the argument does not apply (when $j=0$) to the transposition $(n-1,n)$ of $a$ and to the transposition $(n,1)$ of $b$; and (when $j=1$) does not apply to the transposition $(n-2, n-1)$ of $a$. Nevertheless we can quickly see that not even such transpositions can fix a block. For the last case ($j=1$), assume $B_{n-2}=B_{n-1}$. Then $n-5=(n-2)b$ and $n=(n-1)b$ lie in $B_{n-2}b$; but $a$ fixes $n$ and $n-5$ belongs to a $4$-cycle of $a$, which contradicts what we said above. For the case $j=0$, if $B_{n-1}=B_{n}$ then $B_{n-1}b^2=B_{n-1}$ since $n$ belongs to a transposition of $b$. Then $n-2=(n-1)b^3 \in B_{n-1}b^3=B_{n-1}b$; but also $n-2=(n-1)b^2a \in B_{n-1}b^2a=B_{n-1}$, contradiction. Similarly we can see that $B_n=B_1$ cannot occur. Therefore, we have that, for every $(n,k,\mathfrak s)$, \textit{the support of an $m$-cycle of $a$ or $b$ is contained in one block if and only $m=1$}. \qquad\hspace*{\fill}(1)

In particular, some blocks are made of the points fixed by $a$, and this implies $j \neq 1$, as otherwise a block would consist of 1 point. 

Assume now $a$ has fixed points. Since $j \neq 1$, the last cycles of $a$ and $b$ have not been splitted (that is, we are not in cases (v), (vi)). The set of points fixed by $a$ is some $\{n-t, \ldots, n \}$,  with $\floor{n/k}k+1 \leqslant n-t \leqslant n-1$, and some blocks are made of these points. However, it is easy to check that $b$ takes only the point $n$ outside of this set, hence it cannot preserve these blocks. Therefore, $a$ has no fixed points, and either $j=0$ or $j=2$ and the original $a$ has been multiplied by $(n-1,n)$.

In the latter case (case (iv)), by (1) $B_{n-1} \neq B_{n}$, and a suitable $k$-cycle $c$ of $a$ collaborates to exchange these blocks. Recall now that $a^{-1}b$ is a cycle. Its support contains $k$ and $k+2$, and $kb=k+2$, hence by transitivity of $\gen{a,b}$ the support of $a^{-1}b$ must intersect more than 1 block. In particular, \textit{$a^{-1}b$ fixes 1 block if and only if it fixes it pointwise} (the same argument works also for the case $j=0$).\qquad\hspace*{\fill}(2)

Note that $n$ is fixed by $a^{-1}b$ but $n-1$ is not; hence $c$ is made of $k/2$ points fixed by $a^{-1}b$ and $k/2$ points not fixed by $a^{-1}b$, in such a way that a point fixed by $a^{-1}b$ is followed by a point not fixed by $a^{-1}b$ and viceversa. However, it is easy to see that $a$ does not have a $k$-cycle with this property. Therefore, $j=0$.

In this case, the points not fixed by $a^{-1}b$ are all the $i \equiv 0,1$\: mod $k$, $k+2$ and possibly (case (v)) $n-k/2+1$ and $k/2$. Note that $(k+1)b=k$, and $kb=k+2$. By (1), $B_{k+1} \neq B_k$. If $k+2 \in B_{k+1}$, then $k \neq 3$ as $(4,3,5)$ should intersect the blocks in the same number of points (as noticed at the beginning of the proof). Hence $k+3=(k+2)b \in B_{k+2}b=B_k$, and since $k$ is not fixed by $a^{-1}b$, by (2) we obtain that $k+3$ is not even fixed. Looking at the list of (non)fixed points above, the unique possibility is $n-k/2+1=k+3$: looking mod $k/2$ we get $k=4$ and $n=8$, case that we can check singularly.

Therefore, $B_{k+1}, B_k, B_{k+2}$ are pairwise distinct. Let $k+1 \neq i \in B_{k+1}$. Then by (2) we should have that $i, ib, ib^2$ are pairwise distinct, different from $k+1, k, k+2$ and not fixed by $a^{-1}b$. However, it is easy to check that an $i$ with this property does not exist.

\end{proof}

\subsection{The case $k > n/2$}

\begin{comment}
\begin{thm}
\label{supp}
Let $G$ be a $2$-transitive permutation group of degree $n$ containing a nontrivial element $g$, with $|\text{supp}(g)| < n/4-1$. Then, $G$ contains $A_n$.
\end{thm}
\begin{thm}
\label{jor}
Let $G$ be a primitive permutation group of degree $n$ containing a $q$-cycle, where $q \leqslant n-3$ is prime. Then, $G$ contains $A_n$.
\end{thm}

Ci metto anche l'altro per il momento...

\begin{thm}
\label{asy}
Let $G$ be a $t$-transitive permutation group of degree $n$, with $t > 3 \log n$. Then, $G$ contains $A_n$.
\end{thm}

\begin{thm}
\label{liv}
Let $G$ be a primitive permutation group of degree $n$. If $G$ contains a cycle fixing $m$ points, then $G$ is $(m+1)$-transitive.
\end{thm}

\end{comment}

Here we will apply points (i) and (ii) of Theorem \ref{old}. Moreover, we further divide in two cases. More precisely, for the moment we exclude the case of alternating groups and $k \in \{n-1,n\}$ even. In other words, in the language of Theorem \ref{mainprop} we assume that the following condition, call it (A), holds: $k > n/2$, and whenever $k$ is even and $\mathfrak s=+1$, then $k \leqslant n-2$.

\begin{prof}
The case in which it is possible to choose $k$-cycles (namely, $k$ odd or $k$ even and $\mathfrak s=-1$) were dealt already in \cite M. Anyway, for completeness we exhibit generating pairs here. If $k \leqslant n-1$ we can take $a=(1,2, \ldots, k)$ and $b=(1,2,\ldots, 2k-n, k+1, k+2, \ldots, n)$. In the case $(n,k)=(6,4)$ actually we should change something and take for instance $b=(1,3,5,6)$. If $k=n$, we can take the same $a$ and $b=a(1,2,3)$. We omit the details of the verification since they are straightforward (and very similar to those that we present below).

Now we deal with the case in which it is not possible to choose $k$-cycles, namely, $k$ even and $\mathfrak s=+1$. Consider $a=(1,2, \ldots, k)(k+1, k+2)$ and $b=(1, 2, \ldots, 2k-n, k+1, k+2, \ldots, n)(2k-n+1, 2k-n+2)$. It is easy to check that $\gen{a,b}$ is transitive; now we show that $\gen{a,b}$ is primitive as well.

Assume $\gen{a,b}$ preserves some system of blocks. The transposition of $a$ exchanges 2 blocks otherwise, since $(k+1)b=k+2$, a block would be fixed by both $a$ and $b$. Since the $k$ cycle of $a$ sends $2k-n+1$ to $2k-n+2$, also the transposition of $b$ exchanges these 2 blocks, and since $\gen{a,b}$ is transitive there cannot be other blocks. This implies that $a$ and $b$ do not have fixed points, namely, $k=n-2$. Theorefore, in the case $k \leqslant n-3$ $\gen{a,b}$ is primitive. In the case $k=n-2$, instead, $\gen{a,b}$ preserves the blocks $\{ 1 \leqslant i \leqslant n,\: \text{$i$ even}\}$ and $\{ 1 \leqslant i \leqslant n,\: \text{$i$ odd}\}$. If $n \geqslant 8$, we can multiply $b$ on the right by $(1, 2, 3)$: it is easy to see that, with this modification, $\gen{a,b}$ cannot preserve any system of blocks. For $n=6$, we choose $b=(1,3,5,6)(2,4)$: it works.

Both in the cases $k \leqslant n-3$ and $k=n-2$, we can check that $a^{-1}b$ is a cycle that fixes for instance the point $k+2$. Hence, by Theorem \ref{livin} $\gen{a,b}$ is $2$-transitive. Now direct computation shows that $a^{-1}bab^{-1}=b^ab^{-1}$ moves at most $8$ points. Therefore, by Theorem \ref{old} (i) we are done provided $8 < n/3-2\sqrt n/3$. The cases $n \leqslant 36$ can be checked singularly, so that the proof is complete.

\qed

\end{prof}

Finally, we deal with the remaining case, namely, $k \in \{n-1,n\}$ and we cannot choose $k$-cycles. In other words, the following condition (B) holds: $k \in \{n-1,n\}$ is even, is not a prime power and $\mathfrak s=+1$.

\begin{profo}
Write $k=2^mq$, where $m \geqslant 1$ and $q \geqslant 3$ is odd. Consider the following cycle-shape: $z=(q,2^m, \ldots, 2^m, 2 \ldots, 2)$, where there are $(2^m)$-cycles in number $t \geqslant 1$ and (if $m \geqslant 2$) $2$-cycles in number $0 \leqslant l \leqslant 2^{m-1}-1$, in such a way that $n-1 \leqslant q + 2^mt + 2l \leqslant n$. In other words, we start with a $q$-cycle, then we put as many $2^m$-cycles as we can, and finally we put as many $2$-cycles as we can.

Observe that if $q + 2^{m+1} \leqslant k=q2^m$, that is, $2^{m+1}/({2^{m}-1}) \leqslant q$, then $t \geqslant 2$. Now $2^{m+1}/({2^{m}-1})$ is decreasing for $m \geqslant 2$, and $2^3/(2^2-1) < 3 \leqslant q$, hence $t \geqslant 2$ in this case. If $m=1$, $t \geqslant 2$ unless $q=3$, that is, $k=6$. For the moment we exclude this case, so that we have $t \geqslant 2$.

Now we consider $z$ as a partition of $n$, adding $1$ at the end if necessary. If it defines an odd permutation, we split the first $(2^m)$-cycle into $(2^{m-1})$-cycles (so that we get $1$-cycles in the case $m=1$). Now we put the numbers $\{1, \dots, n\}$ into the cycles, from left to right, obtaining an element of $S_n$ that we call $a$: $a=(1, 2, \ldots, q)(q+1, q+2, \ldots, q+2^m)(\dots)$ (of course, if the first $(2^m)$-cycle has been splitted then $a$ is different). Since $a$ has still at least a $(2^m)$-cycle, it has order $2^mq=k$. Moreover, by construction $a$ lies in $A_n$. We then define $b$ as the element of $A_n$ which has the same cycle-shape of $a$, and which is obtained translating $a$ on the right of $q-2$ positions, that is, $b=(q-1, q, \ldots 2q-2)(2q-1, 2q, \ldots ,2q+2^m-2)(\dots)$, where of course the translation is seen mod $n$.

\begin{comment}

Assume first $m \geqslant 2$. Note that, for every $t_i, h_i, k_i$ natural numbers, $i=1,2$, we have that $q+2^{m-1}t_1+2^mh_1+2k_1 \neq q-2+q+2^{m-1}t_2+2^mh_2+2k_2$ (they are non-equivalent mod $2$). This implies that there do not exist a point $1 \leqslant i \leqslant n$ and cycles $c_1$ and $c_2$ of $a$ and $b$ (respectively), such that $c_1$ and $c_2$ both terminate at $i$. Starting from the first cycle of $a$, this allows to pass back and forth cycles of $a$ and cycles of $b$, until reaching the last cycle of $b$. Since such cycle contains the point $n$, we have that $\gen{a,b}$ is transitive. In the case $m=1$, the same consideration holds provided that the equation above is valid. This is the case unless $t_1=1$ --- a cycle of $b$ may terminate at the point $q+1$. Since we already assumed $q \geqslant 5$, the $q$-cycle of $b$ reaches at least $q+3$, so that also in this case $\gen{a,b}$ is transitive. 
\end{comment}

It is easy to see that $\gen{a,b}$ is transitive. The key point is that $q$ and $2q-2$ are nonequivalent mod $2$, and nothing changes if we add cycles of even length; this allows to pass back and forth cycles of $a$ and cycles of $b$, until reaching the last cycle of $b$. Since such cycle contains the point $n$, we have that $\gen{a,b}$ is indeed transitive.

Assume now that $\gen{a,b}$ preserves some system of blocks, each block having cardinality $t$. Call $c$ the $q$-cycle of $a$, and call $d$ the $q$-cycle of $b$. Note that $c$ cannot fix a block, for $(q-1)b=q$ would imply that $b$ also fixes the block. Therefore, the support of $c$ intersects $r > 1$ blocks, with $r$ divisor of $q$. Since $r$ is odd, it does not divide the length of any other cycle of $a$, hence $c$ permutes all the points of these $r$ blocks. Clearly, the same consideration holds for $d$. Now, employing the notation used in the proof of Claim \ref{clpri}, we have $B_{q-1} \neq B_q$. Hence $B_{q-1}$ and $B_q$ should be contained in the support of both $c$ and $d$. However, the intersection of such supports is $\{q-1,q\}$, and we obtain a contradiction. Therefore, $\gen{a,b}$ is primitive.

There remains to conclude that $\gen{a,b}=A_n$. A suitable power of $a$ is a $q$-cycle, and $q \leqslant n/2$; hence by Theorem \ref{livin} $G$ is $6$-transitive provided $n/2 \geqslant 5$, which is true since we assumed $(n,k) \neq (7,6)$. Now $q \leqslant n/2 < 3n/5$, hence we conclude by Theorem \ref{old} (ii).

In the case $(n,k)=(7,6)$, it is easy to verify that $a=(1,2,3)(4,5)(6,7)$ and $b=(2,3,4)(5,6)(7,1)$ generate $A_7$ (these are exactly the kind of elements we have  used in this proof).

\qed

\end{profo}

\Addresses

\end{document}